\documentclass{amsart}
\usepackage{amssymb}
\usepackage[all]{xy}
\newcommand{\bba}{\mathbb{A}}\newcommand{\bbf}{\mathbb{F}}

\newcommand{\bbp}{\mathbb{P}}\newcommand{\bbr}{\mathbb{R}}

\newcommand{\bbz}{\mathbb{Z}}
\newcommand{\caa}{\mathcal{A}}
\newcommand{\cac}{\mathcal{C}}\newcommand{\cae}{\mathcal{E}}
\newcommand{\cag}{\mathcal{G}}
\newcommand{\cah}{\mathcal{H}}
\newcommand{\cak}{\mathcal{K}}
\newcommand{\cam}{\mathcal{M}}\newcommand{\can}{\mathcal{N}}
\newcommand{\cao}{\mathcal{O}}

\newcommand{\cau}{\mathcal{U}}\newcommand{\cav}{\mathcal{V}}
\newcommand{\cax}{\mathcal{X}}
\DeclareMathOperator{\aut}{Aut}

\DeclareMathOperator{\gal}{Gal}

\DeclareMathOperator{\Ker}{Ker}

\newcommand{\ov}{\overline}

\DeclareMathOperator{\spec}{Spec}
\newfont{\cyrm}{wncysc10}

\DeclareMathOperator{\ur}{ur}
\newtheorem{theorem}[subsection]{Theorem}%\numberwithin{theorem}{section}
\theoremstyle{plain}

\newtheorem{corollary}[subsection]{Corollary}
\newtheorem{lemma}[subsection]{Lemma}
\newtheorem{proposition}[subsection]{Proposition}
\theoremstyle{definition}
\newtheorem{definition}[subsection]{Definition}

\newtheorem{hypothesis}[subsection]{Hypothesis}
\newtheorem{example}[subsection]{Example}
\newtheorem{remark}[subsection]{Remark}
\numberwithin{equation}{subsection}
\begin{document}
\title[Normal subgroups of fundamental group]{Normal subgroups of the
  algebraic fundamental group of affine curves in positive characteristic}
\author[A. Pacheco]{Am\'{\i}lcar Pacheco}
\address{Universidade  Federal do Rio de Janeiro, Instituto de Matem\'atica,
Rua Guaiaquil  83, Cachambi,  20785-050 Rio de Janeiro, RJ, Brasil}
\email{amilcar@acd.ufrj.br}
\author[K. F. Stevenson]{Katherine F. Stevenson}
\address{California State University}
\email{katherine.stevenson@csun.edu}
\author[P. Zalesskii]{Pavel Zalesskii}
\address{Universidade de Bras\'{\i}lia}
\email{pz@mat.unb.br}
\thanks{The first and third authors were partially supported by CNPq research
grants  305731/2006-8 and 307823/2006-7, respectively. They were
also supported by Edital Universal CNPq 471431/2006-0.}
\date{\today}
\begin{abstract}
Let $\pi_1(C)$ be the algebraic fundamental group of a smooth connected affine curve, defined over an algebraically closed field of characteristic $p>0$ of countable cardinality. Let $N$ be a normal (resp. characteristic) subgroup  of $\pi_1(C)$. Under the hypothesis that the quotient $\pi_1(C)/N$ admits an infinitely generated Sylow $p$-subgroup, we prove that $N$ is indeed isomorphic to a normal (resp. characteristic) subgroup of a free profinite group of countable cardinality. As a consequence, every proper open subgroup of $N$  is a free profinite group of countable cardinality.
\end{abstract}
\maketitle
%%%paper%%%
%%%Introduction%%%
\section{Introduction}

Let $C$ be a smooth connected affine curve defined over an
%(?)countable
algebraically closed field $k$ of characteristic $p>0$. Let
$\pi_1(C)$ be its algebraic fundamental group. From the proof of Abhyankar's Conjecture, it is
 known which finite groups are quotients of $\pi_1(C)$ (\textsc{Harbater} \cite[Theorem 6.2]{ha} and \textsc{Raynaud}
\cite[corollaire 2.2.2]{ra}).  However, not much else is known about the
actual structure of the subgroups of $\pi_1(C)$.  The objective of
the paper is to give a description of the normal subgroup
structure of the fundamental group $\pi_1(C)$. We study
the infinite index normal subgroup structure of $\pi_1(C)$ and
prove that the majority of it coincides with the infinite index
normal subgroup structure  of a free profinite group.

For every $g\ge1$ let $P_g(C)$ be the intersection of algebraic fundamental groups $\pi_1(Z)$ of smooth connected affine curves $Z$ which are cyclic \'etale $p$-covers of $C$ and whose genus of its smooth compactification $\bar{Z}$ is at least $g$. This is a normal closed subgroup of $\pi_1(C)$. 

\begin{theorem}\label{mainthm}
Let $\pi_1(C)$ be the algebraic
  fundamental group of a smooth connected affine
curve $C$ defined over an algebraically closed field $k$ of
characteristic $p>0$ of countable cardinality.
Let $N$ be a normal
  (resp. characteristic) subgroup of
$\pi_1(C)$  such that a Sylow $p$-subgroup $M_p$ of the quotient
group $M=\pi_1(C)/N$ is infinitely generated. Suppose furthermore that $N$ is a subgroup of $\pi_1(C)$. Then $N$ is
isomorphic to a normal (resp. characteristic) subgroup of a free
profinite group of countable rank.
\end{theorem}

This gives an almost complete description of such subgroups if one
takes into account the known structure of the normal subgroups of a
free profinite group (\emph{cf.} \cite[chapter 8]{riza}). For
example one deduces immediately from Theorem \ref{mainthm} the
following result.

\begin{corollary}\label{cor1}
Every proper open subgroup of $N$ is a free
profinite group of countable rank.
\end{corollary}

In the proof of the above theorem we use a strong result due to
\textsc{Pop} \cite[Theorem B]{po} which asserts that every finite
\emph{quasi $p$-embedding problem} of $\pi_1(C)$ is properly
solvable (\emph{cf.} also \cite[Corollary 4.6]{hacr}). This result
also implies \textsc{Abhyankar}'s conjecture mentioned above. In
fact, the latter result provides the existence of a proper solution
for a specific finite quasi-$p$ embedding problem.

Unfortunately, \textsc{Abhyankar}'s  conjecture does not hold for
smooth connected projective curves defined over algebraically
closed fields of characteristic $p>0$. Indeed, in the case of a
finite group that is an extension of a finite group of order
prime to $p$ by a finite $p$-group, the two first authors
established in \cite[Theorem 1.3]{past} a necessary and sufficient
condition (\emph{cf.} \cite[condition A]{past}) for
\textsc{Abhyankar}'s conjecture to hold for such a group.
Subsequently this was extended by \textsc{Borne} to the case where
the first group does not necessarily have order prime to $p$,
using modular representation theory  (\emph{cf.} \cite[Theorem
1.1]{bo}).

We also use in the proof of the above theorem the fact that  \textsc{Artin-Schreier}
covers of the projective line can be obtained with arbitrarily high genus. These
covers are necessarily wildly ramified at only one point. Even more,
there is no finite \'etale non-trivial cover of the projective line. So,
the method used in the proof gives no information about the validity of a
similar result in the case of smooth connected projective curves
defined over algebraically closed fields of positive
characteristic. It remains an open problem whether some other strategy
might lead to such a result. However, there is one reason to be skeptical \emph{a
priori}. In contrast with patching
techniques for affine curves,  patching techniques for projective
curves do not in general preserve the base curve.

In fact, the following result was proved in \cite[Theorem
3.1]{st}. Let $G$ be a finite group and suppose $G$ is generated
by two subgroups $G_1$ and $G_2$. Assume that for $i=1,2$ the group $G_i$
is a \textsc{Galois} group over a smooth connected
projective curve $C_i$, \'etale except at one
point $x_i\in C_i$.  Moreover assume that the ramification is
tame over each $x_i$ with cyclic inertia groups generated by $\tau_i \in G_i \subset G$ such that $\tau_1 =\tau_2^{-1}$ in $G$. Then $G$ is realized as an \'etale \textsc{Galois} group
over almost all smooth connected projective curves of genus $g_1+g_2$ where $g_i$
the genus of $C_i$. So, not only is the base
curve not preserved, but also the genus of the curve obtained via
 patching techniques  grows with respect to those of the original curves.

Nevertheless, results similar to Theorem \ref{mainthm}, Corollary
\ref{cor1} and Theorem \ref{homogeneous} were proved for the
algebraic fundamental group of smooth connected projective curves
defined over algebraically closed fields of characteristic zero by
the third author in \cite[Theorem 1.1, Corollary 2.5, Theorem
2.2]{za}, respectively.
Consider then the following three objects:
\begin{itemize}
\item free profinite groups;
\item algebraic fundamental groups of connected projective curves defined over algebraically closed fields of  characteristic zero; and
\item algebraic fundamental groups of connected affine curves defined over algebraically closed fields of countable cardinality and positive characteristic $p>0$;
\end{itemize}
The results mentioned above lead us to think that although these 
three objects are dissimilar in many ways, when we consider the family of projective normal subgroups $N$ of each $\Pi$ such that the quotient 
$\Pi/N$ has an infinitely generated Sylow p-subgroup, this family behaves in exactly the same way in the three cases. 
Therefore we say that their ``\emph{$p$-deep structure}'' are
alike.

In Section 2 we discuss some preliminaries on profinite group
theory, and in Section 3 we discuss a particular family of finite
$p$-groups related to the group $N$. In Section 4 we solve the
embedding problem which suffices to prove Theorem
\ref{homogeneous}. It should be noted that we do not need any
further patching result other than those already present in
\cite{ku}. We acknowledge the author of the preprint as having
constructed an original solution in the case of the commutator
subgroup of $\pi_1(C)$. One of our contributions was to observe
that the \emph{gestalt} of his construction did not require any
additional information on the \textsc{Galois} groups of the
covers. So we were able to replace finite \emph{abelian covers} by
finite \emph{$p$-covers}. It should be noted however that our Corollary \ref{cor1} implies the main result of \textsc{Kumar} \cite[Theorem 4.1]{ku} (\emph{cf.} Corollary \ref{kumar}) and is, in fact, more general than the latter cited result. Additionally, in his thesis \cite[\S 7.3 and Theorem 7.13]{kuth} \textsc{Kumar} comments that his result should hold if one replaces the commutator subgroup by kernel of an epimorphism of $\pi_1(C)$ to an infinitely generated abelian pro-$p$ group. However, this result is also a consequence of Theorem \ref{homogeneous}, and we also show in Example \ref{examku} that Theorem \ref{homogeneous} addresses situations which are not covered by the latter comment in \cite{kuth}. Finally, in Section 5 we obtain group
theoretic consequences from Theorems \ref{melnikov} and
\ref{homogeneous}, and in the Appendix we rewrite the proof of \cite[Theorem 2.2]{za} connecting it with the present paper (\emph{cf.} Theorem \ref{homogeneous}).

\section{Profinite group theory}
For the reader's convenience we review in this section
\textsc{Melnikov}'s characterization of accessible, normal,  and
characteristic subgroups of free profinite groups. We restrict the
discussion to  \emph{second countable} profinite groups
(\emph{i.e.}, groups that have a countable basis of open subsets),
since this is sufficient for our purposes and simplifies the
terminology. Beyond this review, the main function of this section
is to translate the problem of proving that each normal subgroup
$N$ of $\pi_1(C)$ is isomorphic to a normal subgroup of a free
profinite group of countable rank into one of solving finite split
embedding problems for $N$ that arise from those for $\pi_1(C)$
(see Definition \ref{embprob}). This is done in two steps.  First
using a theorem of \textsc{Melnikov} (Theorem \ref{melnikov}
below) we reduce our question to the problem to showing that given finite groups
$\Gamma$ and $G$ the embedding problem $\cae$ defined by
\begin{equation}\label{epn1}
\xymatrix { &N \ar@{>>}[d]^{\psi}\\
\Gamma \ar@{>>} [r]^{\alpha}&G.}
\end{equation}
admits a proper solution under certain additional hypotheses. Then
Lemma \ref{split} allows us, roughly speaking, to assume that
$\cae$ is a split embedding problem.

\subsection{Profinite groups}

\begin{definition}
A closed subgroup $H$ of a profinite group ${\Pi}$ is said to be  {\it
accessible}, if there exists a chain of closed subgroups of  ${\Pi}$
\begin{equation}\label{defacceq}
H={G}_\mu\le \cdots \le {G}_\lambda\le \cdots \le {G}_2\le
{G}_1={\Pi},
\end{equation}
indexed by  the ordinals smaller than a certain ordinal $ \mu $,
such that
\begin{itemize}
\item ${G}_{\lambda+1}\triangleleft {G}_\lambda$ for all ordinals
$\lambda\le \mu$, and
\item if $\nu$ is a limit ordinal such that $\nu\leq\mu$, then
${G}_\nu=\bigcap_{\lambda\le\nu} {G}_\lambda$.
\end{itemize}
\end{definition}

\begin{remark}
Recall that a subgroup $K$ of a group ${\Pi}$ is called \emph{subnormal}
if $K$ is a member of a normal series of ${\Pi}$.
Observe that $H$ is accessible if and only if the image of $H$ is
subnormal in every finite quotient of ${\Pi}$.
\end{remark}

\begin{definition}\label{embprob}
Let ${G}_1$, ${G}_2$ and ${G}_3$ be profinite groups. An \emph{embedding
  problem} $\cae$ consists of a pair of epimorphisms
\begin{equation}\label{defep}
\xymatrix { &{G}_1 \ar@{>>}^{f_1}[d]\\
{G}_2 \ar@{>>}^{f_2} [r]&{G}_3.  }
\end{equation}
The embedding problem $\cae$ is
  \emph{properly solvable}, if there exists an epimorphism
  $f_3:{G}_1\twoheadrightarrow {G}_2$ such that
  $f_2\circ f_3=f_1$, \emph{i.e.}, the
  diagram (\ref{defep}) commutes. The map $f_3$ is called a
  \emph{proper solution} for $\cae$.
\end{definition}

\begin{definition}
For a group $G$, let $M(G)$ denote the intersection of maximal
normal subgroups of ${G}$ and let $R_p(G)$ denote the kernel of the
epimorphism to the maximal pro-$p$ quotient.

If $S$ is a finite simple group, $M_S(G)$ will denote the kernel
of the epimorphism to the maximal direct power of $S$; if $l$ is a prime number and $S$ is a finite simple $l$-group, then $S$ is necessarily isomorphic to the cyclic group of order $l$. Therefore we shall use the notation $M_l(G)$.
 \end{definition}

\begin{definition}
A second countable profinite group ${\Pi}$ is said to be
\emph{homogeneous}, if given any finite groups $\Gamma$ and $G$ and an
embedding problem $\cae$ defined by
\begin{equation}\label{defhomo}
\xymatrix { &{\Pi} \ar@{>>}[d]^f\\
\Gamma \ar@{>>} [r]^{\alpha}&G,  }
\end{equation}
$\cae$ admits a proper solution under the additional hypotheses
that $H:=\Ker(\alpha)$ is  minimal normal and $H\leq M(\Gamma)$.
\end{definition}

The next theorem collects facts originally proved by
\textsc{Melnikov} about homogeneous profinite groups and it also can
be found in \cite[chapter 8]{riza}. More precisely items (1) follows from Theorem 8.5.4, item (2) is a consequence of Corollary 8.6.4 combined with Theorem 8.6.11, item (3) follows from Theorem 8.10.3, and finally item (4) is a consequence of Theorem 8.10.3 combined with Theorem 8.10.2. 

\begin{theorem}\label{melnikov}
Let ${\Pi}$ be a profinite group. Then
\begin{enumerate}
\item  ${\Pi}$ is homogeneous if and only if it is isomorphic to an
accessible subgroup of free profinite group of countable rank;
\item ${\Pi}$ is isomorphic to a normal subgroup of a free profinite
group of countable rank if and only if ${\Pi}$ is homogeneous and
${\Pi}/M_l({\Pi})$ is either trivial or infinite for every prime number $l$;
\item  ${\Pi}$ is isomorphic to a characteristic subgroup of a free
profinite group if and only if ${\Pi}$ is homogeneous and ${\Pi}/M_S({\Pi})$
is either trivial or infinite for every finite simple group
$S$.
\item ${\Pi}$ is free profinite of countable rank if and only if it is
homogeneous and ${\Pi}/M_S({\Pi})$ is infinite for every finite simple
group $S$.
\end{enumerate}
\end{theorem}

We shall also need the following

\begin{lemma}\label{pgplem}
\label{p-group}
Let $P$ be a pro-$p$ group and $\mathcal{G}$ a
finite quotient of $P$ such that its minimal number of generators
$d(\cag)$ is less than the minimal number of generators $d(P)$ of $P$. Then
$\cag\times \bbz/p\bbz$ is a quotient of $P$.
\end{lemma}

\begin{proof}
Let $t$ be a generator of $P$ which is in the kernel $K$ of the
natural epimorphism of $P$ onto $\cag$ and
$\eta:P\longrightarrow\bbz/p\bbz$ an epimorphism onto a group $\bbz/p\bbz$
of order $p$ such that $\eta(t)\neq 1$. Then $P/(K\cap\Ker(\eta))\cong
\cag\times\bbz/p\bbz$ as needed.
\end{proof}

\subsubsection{} Our key result is the following
theorem.

\begin{theorem}\label{homogeneous}
Let $\pi_1(C)$ be the algebraic fundamental group of a smooth
connected affine curve defined over an algebraically closed field of
characteristic $p>0$ of countable cardinality. Let $N$ be a
normal subgroup of $\pi_1(C)$ such that a
Sylow-$p$ subgroup $M_p$ of the quotient group $M=\pi_1(C)/N$ is
infinitely generated. Then $N$ is homogeneous.
\end{theorem}

\subsection{Changing the embedding problem}\label{chep}
It follows from the definition of homogeneous subgroup that in order to prove
Theorem \ref{homogeneous} it suffices to show that given finite
groups $\Gamma$ and $G$ the embedding problem $\cae$ defined by
\begin{equation}\label{epn}
\xymatrix { &N \ar@{>>}[d]^{\psi}\\
\Gamma \ar@{>>} [r]^{\alpha}&G.}
\end{equation}
admits a proper solution under the additional hypotheses that
$H:=\Ker(\alpha)$ is minimal normal and $H\leq M(\Gamma)$. The existence of a proper solution for this embedding problem will be given in Section 4.

The next lemma is an extract of Lemma 2.4 and Example 2.5 (a) in
\cite{hast}.

\begin{lemma}\label{split}
Given a projective profinite group $\Pi$ and a finite embedding
problem
\begin{equation}\label{eppi}
\xymatrix { &\Pi \ar@{>>}[d]^{\psi}\\
\Gamma \ar@{>>} [r]^{\alpha}&G.}
\end{equation}
There exists a finite split embedding problem
\begin{equation}\label{eppi'}
\xymatrix { &\Pi \ar@{>>}[d]^{\psi'}\\
\Gamma' \ar@{>>} [r]^{\alpha'}&G'.}
\end{equation}
such that any proper solution of (\ref{eppi'}) induces a  proper
solution of  (\ref{eppi}). Moreover, if the kernel $H$ of $\alpha$
is of order prime to $p$, then  so is the kernel $H'$ of
$\alpha'$.
\end{lemma}

\begin{proof} Since $\Pi$ is projective there is a weak solution
$\psi':\Pi \longrightarrow \Gamma$ of (\ref{eppi}). Let
$G':=\psi'(\Pi)$. Let $\Gamma':=\Gamma\times_GG'$ be the pullback.
Then we have the following commutative diagram
\begin{equation}\label{split1}
\xymatrix {
&\Pi \ar@{>>}[d]_{\psi'}\ar@{>>}@/^1pc/[dd]^{\psi}\\
\Gamma' \ar@{>>} [r]^{\alpha'}\ar[d]&G'\ar[d]\\
 \Gamma \ar@{>>}[r]^{\alpha}&G.
}
\end{equation}
Since $\Gamma'$ is a subgroup of a direct product
$\Gamma\times G'$, the diagram (\ref{split1}) yields the following
split embedding problem
\begin{equation}\label{split2}
\xymatrix { &\Pi \ar@{>>}[d]^{\psi'}\\
\Gamma' \ar@{>>} [r]^{\alpha'}&G'.}
\end{equation}

Given a proper solution $\lambda': \Pi \twoheadrightarrow\Gamma'$ of
 (\ref{split2}), it is automatic that $\lambda:=q\circ\lambda'$ is a
 proper solution of (\ref{epn}), where
 $q:\Gamma'\twoheadrightarrow\Gamma$ is the projection to the first
 coordinate.

The last statement follows from the fact
 that it is always possible to factor  $\alpha'$ through  the quotient of  $H'$ by the normal
 closure of a Sylow $p$-subgroup of $H'$.
\end{proof}

\begin{remark}\label{suffsp}
It follows from Lemma \ref{split} that in order to properly solve
the embedding problem (\ref{epn}), it suffices to properly solve
the associated finite split embedding problem
\begin{equation}\label{epn'}
\xymatrix { &N \ar@{>>}[d]^{\psi'}\\
\Gamma' \ar@{>>} [r]^{\alpha'}&G'.}
\end{equation}
(\emph{Warning}: $H'$ is not a minimal normal subgroup of
$\Gamma'$).
\end{remark}

\begin{remark}\label{remdevi}
By \cite[Lemma 8.3.8]{riza} there exists  an open  normal subgroup
$R$ of $\pi_1(C)$ containing $N$ and an epimorphism
$\varphi:R\twoheadrightarrow G$ such that $\varphi_{|N}=\psi$.
Note that every $R$ is the fundamental group of a finite \'etale
cover $C'$ of $C$. So after replacing $\pi_1(C)$ by $R=\pi_1(C')$,
if necessary,  we may assume the existence of the following
commutative diagram:
\begin{equation}\label{epch}
\xymatrix {& N\ \ar@{>>}[d]^{\psi}\ar@{>->}[r]&\pi_1(C)\ar@{>>}[dl]^{\varphi}\\
              \Gamma \ar@{>>} [r]^{\alpha}&G.&  }
\end{equation}
It follows from (\ref{epch}) that for every subgroup $N\leq
L\leq \pi_1(C)$, the restriction $\varphi_{|L}$ of $\varphi$ to $L$
is an epimorphism.
\end{remark}

\section{A family of $p$-groups}\label{fampcov}

\emph{Terminology.}  For any ring $R$, let $\text{frac}(R)$ be its total ring of fractions. If $R$ is a domain and $R\subset S$ is an extension of rings, then we say that $S$ is \emph{generically separable} as an $R$-algebra if $\text{frac}(S)$ is a separable $\text{frac}(R)$-algebra and if no non-zero element of $R$ becomes a zero-divisor of $S$. For any finite group $G$, a $G$-\textsc{Galois} $R$-algebra consists of a finite separable $R$-algebra $S$ together with a group homomorphism $\rho:G\to\aut_R(S)$ with respect to which $G$ acts simply transitively on the generic geometric fiber of $\spec(S)\to\spec(R)$. We say that a morphism of schemes $\phi:Y\to X$ is \emph{generically separable} (resp. $G$-\textsc{Galois} with respect to a homomorphism $\rho:G\to\aut_X(Y)$) if $X$ can be covered by affine open subsets $U=\spec(R)$ such that the ring extension $R\subset\cao(\phi^{-1}(U))$ have the corresponding property. A morphism $\phi:Y\to X$ which is finite and generically separable (resp. finite and $G$-\textsc{Galois}) will be called a \emph{cover} (resp. $G$-\textsc{Galois} \emph{cover}).

\subsection{A geometric translation}
The results of this section reduce our problem to a problem in algebraic geometry.  Specifically, given a $G$-\textsc{Galois} cover of a curve $C$ arising from an \emph{a priori} specified epimorphism $\varphi: \pi_1(C) \twoheadrightarrow G$ such that $\varphi|_N$ is still an epimorphism, we want to find a $\Gamma$-\textsc{Galois} cover of $C$ that dominates the given $G$-cover.  This $\Gamma$-\textsc{Galois} cover will correspond to an epimorphism $\lambda: \pi_1(C) \twoheadrightarrow \Gamma$ and we must show that $\lambda|_N$ is an epimorphism.

One of the crucial points in the construction of the aforementioned $\Gamma$-\textsc{Galois} cover is to produce finite \'etale covers of $C$ which admit as subcovers finite \'etale \textsc{Galois} covers whose \textsc{Galois} groups are  $p$-groups. These finite $p$-groups will not be arbitrary, in fact they are chosen as finite quotients of a \textsc{Sylow}-$p$ subgroup $M_p$ of $M:=\pi_1(C)/N$. In this section we focus on these groups by constructing the covers they determine. This is done in the language of function fields (see diagram (\ref{pdiag1})). For a complete description of our global strategy, see Subsection \ref{subsecstra}.

\medskip

Let $\ov{C}$ be a smooth compactification of $C$ and
$S:=\ov{C}\setminus C$. Let $K:=k(C)$ be the function field of $C$ and
$K_{\ur,S}$ the maximal \textsc{Galois} extension of $K$ (inside some
algebraic closure of $K$) which is unramified over the points of
$C$. The fundamental group $\pi_1(C)$ can also be viewed as the
\textsc{Galois} group $\gal(K_{\ur,S}/K)$.
\medskip

\textit{Notation.} For any smooth connected curve $\cac$ defined over
$k$, we will denote by $k(\cac)$ its function field. Given a profinite
group $\Pi$ we will denote by $d(\Pi)$ the minimal number of
generators of $\Pi$.
\medskip

We start by observing that since $N$ is a closed subgroup of $\pi_1(C)$, by infinite
\textsc{Galois} theory there exists a unique subextension $K_M/K$ of
$K_{\ur,S}/K$ such that $K_M$ is field fixed by $N$. Moreover, since
$N$ is normal in $\pi_1(C)$, then $K_M/K$ is \textsc{Galois} with
$\gal(K_M/K)=M$. Similarly, since $M_p$ is closed in $M$, denote by
$K_p$ the subfield of $K_M$ fixed by $M_p$.

The \textsc{Sylow}-$p$ subgroup $M_p$ of $M$ can be described via the
following projective limit $M_p=\varprojlim\Upsilon$, where $\Upsilon$
runs through the open subgroups of $M$ containing $M_p$. Observe that
since $M=\gal(K_M/K)$, each $\Upsilon$ is equal to $\gal(K_M/k(U_{\Upsilon}))$,
where $U_{\Upsilon}$ is a finite \'etale cover of $C$. Alternatively,
$M_p=\varprojlim\Upsilon^{(p)}$, where
$\Upsilon^{(p)}$ denotes the maximal pro-$p$ quotient of
$\Upsilon$. Let $K_{U,p}/k(U_{\Upsilon})$ be the subextension of $K_M/k(U_{\Upsilon})$ such
that $\gal(K_{U,p}/k(U_{\Upsilon}))=\Upsilon^{(p)}$. Furthermore, any finite
quotient $\cag$ of $M_p$ is indeed
a finite quotient of some $\Upsilon^{(p)}$. Hence, $\cag$ will be
equal to $\gal(k(U_p)/k(U_{\Upsilon}))$ for some finite \'etale \textsc{Galois}
$p$-cover $U_p$ of $U_{\Upsilon}$. The following diagram summarizes the
situation:
\begin{equation}\label{pdiag1}
\xymatrix{
K_{\ur,S}\ar@{-}[dr]^N\ar@{-}[dddddrrr]_{\pi_1(C)}\\
&K_M\ar@{=}[r]\ar@{-}[ddddrr]^M&K_M\ar@{=}[r]\ar@{-}[dddr]_{\Upsilon}&K_M\ar@{-}[d]^{M_p}\ar@{-}[drrr]\\
&&&K_p\ar^@{-}[dd]&&&K_{U,p}\ar@{-}[d]\ar@{-}[ddlll]_{\Upsilon^{(p)}}\\
&&&&&&k(U_p)\ar@{-}[dlll]^{\cag}\\
&&&k(U_{\Upsilon})\ar@{-}[d]\\
&&&k(C).
}
\end{equation}

\subsection{Description of the family}\label{descfam}

We consider the family of finite groups which are quotients of
$M_p$. By (\ref{pdiag1}) they are the \textsc{Galois} groups
$\gal(k(U_p)/k(U_{\Upsilon}))$ of finite \'etale \textsc{Galois} $p$-cover $U_p$
of a finite \'etale cover $U_{\Upsilon}$ of $C$ such that $\gal(K_M/k(U_{\Upsilon}))=\Upsilon$.

\begin{remark}\label{bigprk}
Note that $d(M_p)\ge d(\Upsilon^{(p)})\ge d(\cag)$. Since $M_p$ is
infinitely generated, after eventually taking a
$\Upsilon'\supset\Upsilon$, and repeating the same previous
construction for $\Upsilon'$, we may assume that
$d(\cag)<d(\Upsilon^{(p)})$.
\end{remark}

Observe that the diagram (\ref{pdiag1}) implies the following
inclusions between fundamental groups and the profinite group $N$ :
\begin{equation}\label{fundinc}
\pi_1(C)\supset\pi_1(U_{\Upsilon})\supset\pi_1(U_p)\supset N.
\end{equation}

\subsection{Explaining the strategy}\label{subsecstra}
We start with the embedding
problem (\ref{epn}). Next we deduce from it the associated split
embedding problem (\ref{epn'}). Then the \emph{d\'evissage} of Remark
\ref{remdevi} and the inclusions (\ref{fundinc}) produce the following
diagram of split embedding problems
\begin{equation}\label{finaldiag1}
\xymatrix {& N\
  \ar@{>>}[d]^{\psi}\ar@{>->}[r]&\pi_1(U_p)\ar@{>>}[dl]_{\varphi'_1}
\ \ar@{>->}[r]&\pi_1(C)\ar@{->>}[dll]_{\varphi'}\\
              \Gamma' \ar@{>>} [r]^{\alpha'}&G',& }
\end{equation}
where $\varphi'_1:=\varphi'_{|\pi_1(U_p)}$. In Section \ref{secsol}, we produce a
curve $U_p$ (which will be actually denoted as $T'$ in the sequel)
such that the split embedding problem given by the pair of maps
$(\alpha',\varphi'_1)$ admits a proper solution. Then by Lemma
\ref{split} this will induce a proper solution to the embedding
problem defined by the pair of maps $(\alpha,\varphi_{|\pi_1(T')})$. As
a final step, the group theoretic argument of Subsection \ref{gpthar}
will produce from this proper solution, a proper solution to the
original embedding problem (\ref{epn}).

\section{Solving the embedding problem (\ref{epn})}\label{secsol}

%\subsection{Solving the embedding problem (\ref{ept1})}\label{ssecsol}

We have already observed in Section \ref{fampcov} that in order to
obtain a proper solution for the embedding problem (\ref{epn}), it
suffices to produce a finite \'etale \textsc{Galois} $p$-cover $T'\to
U_{\Upsilon}$ such that $\gal(k(T')/k(U_{\Upsilon}))$ is a finite quotient of $M_p$ and that
the embedding problem in (\ref{finaldiag1}) defined by the pair of
maps $(\alpha',\varphi'_1)$ admits a proper solution. Thus, in
Subsection \ref{embT} we use patching techniques to build a curve $T$
over $k[[t]]$ for which we can solve $(\alpha',\varphi'_1)$.   In Subsection
\ref{special} we specialize this solution to obtain a solution over
$k$.  Then in Subsection \ref{gpthar} we go on to use a group
theoretic argument to solve (\ref{epn}). Observe additionally that the proof of the existence of a proper solution of the latter embedding problem will complete the proof of Theorem 2.9. This geometric proof does not require that the field $k$ is countable, it works well for any algebraically closed field $k$ of characteristic $p>0$.

\begin{hypothesis}\label{hypB}
From this point through Subsection \ref{gpthar} our goal is to solve the
embedding problem (\ref{epn}).  It follows from  Lemma \ref{split}
that it suffices to solve split embedding problem (\ref{epn'}).
So we suppose we are already in the latter case. We also assume that $H$ is of order prime
to $p$. As already mentioned above this implies that the kernel $H'$ of the latter embedding problem also has order prime to $p$.
\end{hypothesis}

\begin{remark}
We actually do not need the hypothesis that $H$ has order prime to $p$ for Subsection \ref{gpthar}. Moreover the case complementary
to it  is a consequence of the fact that every finite quasi
$p$-embedding problem of the algebraic fundamental group of a
smooth connected affine curve $C$ defined over an algebraically
closed field of characteristic $p>0$ is properly solvable
(\emph{cf.} \cite[Corollary 4.6]{hacr} and \cite[Theorem B]{po}).
The argument is given in Subsection \ref{p1sec}.
\end{remark}

\begin{remark}The main idea of our construction is the following. We start on one side with a $G'$-\'etale \textsc{Galois} cover of one the curves $U_p$ described above. Recall that these curves satisfy the property that there exists an \'etale \textsc{Galois} cover $U_p/U_{\Upsilon}$ whose \textsc{Galois} group equals a finite quotient of $M_p$. On the other side we consider a $H'$-\'etale \textsc{Galois} cover of an \textsc{Artin-Schreier} cover of $\bbp^1_k$. Remember that our goal is to solve the split embedding problem (\ref{epn'}) and that $H'$ is exactly the kernel of this embedding problem. After deforming these covers they are patched so as to produce a $\Gamma'$-\textsc{Galois} cover $W$ of a curve $T$ over $\spec(k[[t]])$. Taking the generic fiber of $W$ and then specializing it, we will obtain a $\Gamma'$-\textsc{Galois} cover of a curve $T'$ defined over $k$ and this $T'$ will be one of the $U_p$'s previously defined.\end{remark}

\subsection{Solving the embedding problem for some $T$.} \label{embT}
\subsubsection{Description of the ${G'}$-cover}\label{gcov}

Let $\cag$ be a non-trivial finite quotient of $M_p$ and let $\phi_{U_p}:U_p\to U_{\Upsilon}$ be the
finite \'etale
\textsc{Galois} $p$-cover of $U_{\Upsilon}$ such that $\gal(U_p/U_{\Upsilon})=\cag$, where
$U_{\Upsilon}$ is a finite \'etale cover of $C$ satisfying $\gal(K_M/k(U_{\Upsilon}))=\Upsilon$ and $\Upsilon$ is an open subgroup of $M$ containing $M_p$ (\emph{cf.} diagram (\ref{pdiag1}). The inclusion of fundamental groups (\ref{fundinc}) and the argument presented in subsection \ref{subsecstra} show us that there
exists an epimorphism $\varphi_p:\pi_1(U_p)\twoheadrightarrow {G'}$ such that
$\varphi_p=\varphi_{|\pi_1(U_p)}$ (\emph{cf.} (\ref{epch})). This means that
there exists an \'etale \textsc{Galois} cover $V\to U_p$ with
${G'}=\gal(V/U_p)$.

Let $\ov{U}_{\Upsilon}$ be a smooth compactification of $U_{\Upsilon}$. Let $X$ be the
normalization of $\ov{U}_{\Upsilon}$ in $k(U_p)$ and $\phi_X:X\to\ov{U}_{\Upsilon}$ the
normalization morphism. The morphism $\phi_X$ extends $\phi_{U_p}$ to
$X$, hence it is $\cag$-\textsc{Galois} and \'etale over $U_{\Upsilon}$.
Similarly, let $W_X$ be the normalization of $X$ in $k(V)$ and
$\phi_{W_X}:W_X\to X$ the normalization morphism. This is a
${G'}$-\textsc{Galois} cover \'etale over $\phi_X^{-1}(U_{\Upsilon})$.

\subsubsection{Description of the $H'$-cover}\label{hcov}

Let $\phi_Y:Y\to\bbp^1_y$ be an \textsc{Artin-Schreier} cover
totally ramified over $\infty$. Denote its equation by
$z^p-z=f(y)$, where $f\in k[z]$ has degree $\nu$ relatively prime
to $p$. The genus of $Y$ can be arbitrarily high, since it is
equal to $(p-1)(\nu-1)/2$. For a discussion of other cases of
wildly ramified covers with arbitrarily large genus \emph{cf.}
\cite{pr}. Thus, we may suppose that the genus $g(Y)$ of $Y$ is at
least $\max(2,r)$, where $r$ denotes the minimal number of
generators  of $H'$. Since $H'$ has order prime to $p$, it follows
from \cite[Exp. X, Corollaire 3.10]{sga1} that there exists an
\'etale $H'$-\textsc{Galois} cover $\phi_{W_Y}:W_Y\to Y$. We reparametrize $\bbp^1_y$ so that $Y$ is totally ramified over $y=0$. Denote by $s\in Y$ the point lying over $y=0$.

\subsubsection{Descending to $\bbp^1$}

Since $k$ is algebraically closed, then $k(U_{\Upsilon})/k$ has a separating
variable. Moreover, a stronger version of \textsc{Noether}'s
normalization theorem \cite[Corollary 16.18]{ei} implies that
there exists a finite proper $k$-morphism $\phi_{U_{\Upsilon}}:U_{\Upsilon}\to\bba^1_x$
which is furthermore generically separable.  The branch locus of
the morphism will be of codimension 1, so it will be \'etale away
from finitely many points. After a reparametrization, if
necessary, we may assume that none of these points lie over
$x=0$. Moreover, $\phi_{U_{\Upsilon}}$ extends to a finite proper morphism
$\phi_{\ov{U}_{\Upsilon}}:\ov{U}_{\Upsilon}\to\bbp^1_x$. Note that
$\phi_{\ov{U}_{\Upsilon}}^{-1}(\{x=\infty\})=\ov{U}_{\Upsilon}\setminus U_{\Upsilon}=:\cau_{\Upsilon}$. Composing the latter morphism with the $\cag$-\textsc{Galois} cover $\phi_X:X\to\ov{U}_{\Upsilon}$ we obtain 
$\phi'_X:=\phi_X\circ\phi_{\ov{U}_{\Upsilon}}:X\to\bbp^1_x$. This morphism is
\'etale at the points of the set
$\{r_1,\cdots,r_n\}:=\phi_X'^{-1}(\{x=0\})$.

\subsubsection{Patching and deformation}
Now we must patch together and deform the $G'$-\textsc{Galois} cover $\phi_{W_X}:W_X\to X$ and the $H'$-\textsc{Galois} cover $\phi_{W_Y}:W_Y\to Y$. 
Let $R$, resp. $A$, be rings such that
$\spec(R)=X\setminus\phi_X^{\prime{-1}}(\{x=\infty\})$, resp.
$\spec(A)=Y\setminus\phi_Y^{-1}(\{y=\infty\})$. Hence,
$k[x]\subset R$ and $k[y]\subset A$. Let $F:=(R\otimes_k
A\otimes_kk[[t]])/(t-xy)$ and $T^f:=\spec(F)$. Let $T$ be the
\textsc{Zariski} closure of $T^f$ in
$Z:=X\times_{\spec(k)}Y\times_{\spec(k)}\spec(k[[t]])$.

Observe
that under the first, resp. second, projection of $T$ inside $Z$,
$T$ is an $X$-scheme, resp. an $Y$-scheme. Let $W_{XT}$, resp.
$W_{YT}$ be the normalization of an irreducible dominating component of
$W_X\times_XT$, resp. $W_Y\times_YT$.

We next state the items of \cite[Proposition 6.3]{ku} which we will need and give a sketch of their proof. For a more detailed description see \emph{loc. cit.}. We start with the following notation.

\emph{Notation.} Let $\cax$ be a scheme over $\spec(k[[t]])$. We denote by $\tilde{\cax}$ its completion along its closed fiber, \emph{i.e.}, its $(t)$-adic completion. Let $G_1$ be a finite group and $G_2$ a subgroup of $G_1$. Given a $G_2$-\textsc{Galois} cover $Z'\to Z$ of smooth connected curves defined over $k$, it induces a disconnected $G_1$-\textsc{Galois} cover $\text{Ind}^{G_1}_{G_2}(Z')\to Z$. The upper curve consists of as many copies of $Z'$ as the index of $G_2$ in $G_1$. For any scheme $X$ and finite group $\cah$, let $\cah\can(X)$ be the category of generically separable coherent locally free sheaves of $\cao_X$-algebras $S$ endowed with a $\cah$-action which is transitive on the geometric generic fibers of $\spec(S)\to X$.  

\emph{Further notation.} Let $L:=\spec(k[z])$ be an affine line. The $k$-algebra homomorphism $k[[t]][z]\to F$ defined by $z\mapsto x+y$ induces a $\spec(k[[t]])$-morphism $\phi:T^f\to L^*:=L\times_{\spec(k)}\spec(k[[t]])$. Let $\lambda\in L$ be the closed point corresponding to the point $z=0$. Since $L$ is included in $L^*$ as the closed fiber, $\lambda$ can be viewed as a closed point of $L^*$ corresponding to the maximal idea $(z,t)$ in $k[[t]][z]$.  Denote by $\{\tau_1,\cdots,\tau_n\}$ the inverse image $\phi^{-1}(\lambda)$. The closed fiber of $T$ is a reducible curve consisting of $X$ and $n$ copies of $Y$, and each of these copies of $Y$ intersects $X$ at one of the points $\{\tau_1,\cdots,\tau_n\}$.  (This is because the locus of $t=0$ is the same as that of $xy=0$ in $T$, and the locus of $t=0$ and $x+y=0$ in $T$ is the same as that of $x=0$ and $y=0$.) Thus, we can identify the points $\{\tau_1,\cdots,\tau_n\}$ with the points $\{r_1,\cdots,r_n\}$ on the $X$-part of the closed fibre. Notice that on the ``$i$-th copy" of $Y$, the point $\tau_i$ corresponds to $s\in Y$, the unique point lying over $y=0$. Therefore, we actually consider $\{\tau_1,\cdots,\tau_n\}$ to be points lying in $T$. Let $T^0:=T\setminus\{\tau_1,\cdots,\tau_n\}$, $T_X:=T^0\setminus\{x=0\}$ and $T_Y:=T^0\setminus\{y=0\}$. For each $i=1,\cdots,n$, let $\hat{T}_i:=\spec(\hat{\cao}_{T,\tau_i})$. For each $i=1,\cdots,n$, let $\hat{K}_{X,r_i}:=\text{frac}(\hat{\cao}_{X,r_i})$, $\cak_{X,r_i}:=\spec(\hat{K}_{X,r_i}[[t]]\times_{\spec(k[y])}\cao_{Y,s})$, $\hat{K}_{Y,s}:=\text{frac}(\hat{\cao}_{Y,s})$ and $\cak_Y^i:=\spec(\hat{K}_{Y,s}[[t]]\otimes_{\spec(k[x])}\cao_{X,r_i})$.

One of the ingredients of the proposition is the following lemma proved in \cite[Lemma 6.2]{ku} which is analogous to \cite[Corollary 2.2]{ha}.

\begin{lemma}\label{catlem}For any finite group $\cah$, the base change functor
$$\cah\can(T)\to\cah\can(\widetilde{T_X}\cup\widetilde{T_Y})\times_{\cah\can(\bigcup_{i=1}^n
\cak_{X,r_i}\cup\cak_{Y}^i))}\cah\can\left(\bigcup_{i=1}^n\hat{T}_i\right)$$
is an equivalence of categories.
\end{lemma}

\begin{proposition}\label{ku6.3}\cite[Proposition 6.3]{ku} Let $\phi_{W_X}:W_X\to X$ be an irreducible normal $G'$-cover which is \'etale over $r_1,\cdots,r_n$. Similarly, take $\phi_{W_Y}:W_Y\to Y$ to be an irreducible normal $H'$-cover which is \'etale over $s$. Let $W_{XT}$, resp. $W_{YT}$, be the normalization of an irreducible dominating component of $W_X\times_XT$, resp. $W_Y\times_YT$. Then there exists an irreducible normal $\Gamma'$-cover $W\to T$ satisfying the following conditions
\begin{enumerate}
\item $W\times_T\widetilde{T_X}=\mathrm{Ind}^{\Gamma'}_{G'}(\widetilde{W_{XT}\times_TT_X})$.
\item $W\times_T\widetilde{T_Y}=\mathrm{Ind}^{\Gamma'}_{H'}(\widetilde{W_{YT}\times_TT_Y})$.
\item $W/H'\cong W_{XT}$ as an \'etale $H'$-cover of $T$.
\end{enumerate}\end{proposition}

\begin{proof}[Sketch of proof] Denote the right hand sides of (1) and (2) in the proposition by $\widetilde{W_X}$ and $\widetilde{W_Y}$, respectively. The structural sheaf of their union $\widetilde{W^0}$ is an object of $\Gamma'\can(\widetilde{T_X}\cup\widetilde{T_Y})$. Using the fact that $\phi_{W_X}$ is \'etale over each $r_i$ and $\phi_{W_Y}$ is \'etale over $s$, one concludes that the restriction of $\widetilde{W^0}$ to $\bigcup_{i=1}^n(\cak_{X,r_i}\cap\cak_Y^i)$ is a $\Gamma'$-cover induced by the trivial cover. Denote by $\hat{W}_i$ the $\Gamma'$-cover of $\hat{T}_i$ induced by the trivial cover. The structural sheaf of $\hat{W}:=\bigcup_{i=1}^n\hat{W}_i$ is an object of $\Gamma'\can(\bigcup_{i=1}^n\hat{T}_i)$. Moreover, the restriction of $\hat{W}$ to $\bigcup_{i=1}^n(\cak_{X,r_i}\cup\cak_Y^i)$ is a $\Gamma'$-cover induced by the trivial cover. After having fixed an isomorphism between two of these induced covers one can apply Lemma \ref{catlem}. Then (1) and (2) follow immediately from this result. Finally to obtain (3), it suffices to note that both $W/H'$ and $W_{XT}$ restrict to the same $G'$-covers over the patches $\widetilde{T_X}$, $\widetilde{T_Y}$ and $\hat{T}_i$, and then apply Lemma \ref{catlem} again. \end{proof}

Taking $\phi_{W_X}:W_X\to X$ and $\phi_{W_Y}:W_Y\to Y$ to be as defined in Subsections 4.4.1-4.4.3, and using Proposition \ref{ku6.3}, we obtain an irreducible normal $\Gamma'$-cover $W\to T$ satisfying the conditions above.   It remains to specialize this situation so that (as stated in the introduction to Section 4) we will have produced a finite \'etale \textsc{Galois} $p$-cover $T'\to U_{\Upsilon}$ such that $\gal(k(T')/k(U_{\Upsilon}))$ is a finite quotient of $M_p$ and that the embedding problem in (\ref{finaldiag1}) defined by the pair of maps $(\alpha',\varphi'_1)$ admits a proper solution. 

\subsubsection{The generic fiber}

For any $k[[t]]$-scheme $\cav$, let $\cav^g$ be its generic fiber
over $k((t)):=\text{frac}(k[[t]])$. Denote by $B$ the locus of
$xy=t$ in
$\ov{U}_{\Upsilon}\times_{\spec(k)}\bbp^1_y\times_{\spec(k)}\spec(k[[t]])$.
Let $\phi_X:X\to\ov{U}_{\Upsilon}$, resp. $\phi_Y:Y\to\bbp^1_y$, be the cover
of Subsubsection \ref{gcov}, resp. \ref{hcov}. The restriction of
the map
\begin{multline*}
\phi_X\times\phi_Y\times\text{id}:X\times_{\spec(k)}Y\times_{\spec(k)}
\spec(k[[t]])\longrightarrow\\
\ov{U}_{\Upsilon}\times_{\spec(k)}\bbp^1_y\times_{\spec(k)}\spec(k[[t]])
\end{multline*}
to $T$ gives a morphism $T\to B$.
%\newpage

The following diagram summarizes the situation
\begin{equation}\label{bigdiag}
\xymatrix{
&&W\ar[dl]_{H',\text{ \'etale}}\ar[dd]_{{\Gamma'}}\ar[dr]\\
&W/H\cong W_{XT}\ar[ddl]\ar[dr]^{G'}&&W_{YT}\ar[ddr]\ar[dl]\\
&&T\subset Z\ar@/^/[ddl]\ar@/_/[ddr]\ar[dd]\ar[dl]\ar[dr]\\
W_X\ar[dr]^{G'}&X\times_k\spec(k[[t]])\ar[d]&&Y\times_k\spec(k[[t]])\ar[d]&W_Y\ar[dl]_{H'}\\
&X\ar[d]_{\cag,\phi_X}&B\ar[dr]\ar[dl]&Y\ar[d]^{\phi_Y,\bbz/p\bbz}\\
&\ov{U}_{\Upsilon}\ar[d]&&\bbp^1_y\\
&\bbp^1_x.
}
\end{equation}

\begin{lemma}\label{lemku6.4}(\emph{cf.} \cite[Lemma 6.4]{ku})
Let $X$, $Y$ and $T$ be as in (\ref{bigdiag}). Denote
$\cag':=\cag\times\bbz/p\bbz$. The morphism $T\to B$
yields a $\cag'$-\textsc{Galois}
cover $T^g\to\ov{U}_{\Upsilon}\times_{\spec(k)}\spec(k((t)))$.
\end{lemma}

\begin{proof}We observe first that $B^g=\ov{U}_{\Upsilon}\times_{\spec(k)}\spec(k((t)))$. Next, 
it follows from the definition of $T$ that the function field of its
generic fiber $T^g$ over $k((t))$ is the compositum of
$L_1:=k(X)\otimes_kk((t))$ and $L_2:=k(Y)\otimes_kk((t))$. But these
fields are linearly disjoint over $k((t))$. In fact, $k(X)$ is a finite
extension of $k(x)$ and $k(Y)$ is a finite extension of $k(t/x)$
(recall that $y=t/x$). Hence
$(k(X)\otimes_kk((t)))\cap(k(Y)\otimes_kk((t)))=k((t))$. The lemma now
is a consequence of \cite[Chapter 2, Lemma 2.5.6]{frja}.
\end{proof}

Recall that $\cau_{\Upsilon}=\ov{U}_{\Upsilon}\setminus U_{\Upsilon}$. Since the branch locus of
$W^g\to T^g$ is determined by that of $W\to T$, then $W^g\to T^g$
is ramified only at the points points of $T^g$ lying
above those in $\cau_{\Upsilon}\times_{\spec(k)}\spec(k((t)))$. Moreover,
$T^g\to\ov{U}_{\Upsilon}\times_{\spec(k)}\spec(k((t)))$ is ramified only at the
points of $\cau_{\Upsilon}\times_{\spec(k)}\spec(k((t)))$. So, we get a sequence of
coverings with the following groups and ramification behavior :
%\newpage

\begin{equation}\label{genpic}
\xymatrix{
W^g\ar[rdd]_{{\Gamma'}}\ar@{=}[r]&W^g\ar[d]^{H',\text{ \'etale}}\\
&W^g/H\cong W_{XT}^g\ar[d]^{{G'},\text{ branched only at points lying over
  }\cau_{\Upsilon}\times_{\spec(k)}\spec(k((t)))}\\
&T^g\ar[d]^{\cag',\text{ branched only at }\cau_{\Upsilon}\times_{\spec(k)}\spec(k((t)))}\\
&\ov{U}_{\Upsilon}\times_{\spec(k)}\spec(k((t))).}
\end{equation}

\subsubsection{Specialization}\label{special}
In order to obtain a similar sequences of coverings over the field $k$ with the same ramification behavior we need the following specialization result. 

\begin{proposition}\label{ku6.6}\cite[Proposition 6.6]{ku} Let $k$ be an algebraically closed field and $X_0,\cdots,X_3$ normal projective curves defined over $\spec(k[[t]])$. Denote by $X_i^g$ the corresponding generic fibers. Assume the following conditions are satisfied:
\begin{enumerate}
\item For $i=1,2,3$, the curve $X_i$ has generically smooth closed fiber denoted by $X_i^0$.
\item For $i=1,2,3$, there exists proper surjective morphisms $\psi_i:X_i\to X_{i-1}$ defined over $\spec(k[[t]])$. Denote by $\psi^g_i:X_i^g\to X_{i-1}^g$ the induced morphisms on the generic fibers. 
\item The morphisms $\psi_i$ from part (2) are \textsc{Galois} covers such that the group corresponding to $\psi_1$ is $\cag'$, to $\psi_2$ is $G'$ and to $\psi_3$ is $H'$. The composite morphism $\psi_3\circ\psi_2$ is also a \textsc{Galois} cover with group $\Gamma'$
\item There exists a smooth projective curve $Y_0$ defined over $k$ such that $X_0^g=Y_0\times_{\spec(k)}\spec(k((t))$. 
\item The curve $X_1^g$ has genus at least 1.
\item Let $\zeta_1,\cdots,\zeta_r\in Y_0$ and $\xi_1,\cdots,\xi_r\in X_0^g$ be the points induced by them. Suppose that the composite morphism $\psi_3^g\circ\psi_2^g\circ\psi_1^g$ is \'etale away from $\{\xi_1,\cdots,\xi_r\}$.
\end{enumerate}
Then there exist smooth connected projective curves $Y_1,Y_2,Y_3$ defined over $k$ and morphisms $\eta_i:Y_i\to Y_{i-1}$ for $i=1,2,3$ such that their composite $\eta_3\circ\eta_2\circ\eta_1$ is \'etale away from $\{\zeta_1,\cdots,\zeta_r\}$. Finally the morphisms $\eta_1,\eta_2,\eta_3,\eta_3\circ\eta_2$ are \textsc{Galois} covers with groups $\cag'$, $G'$, $H'$ and $\Gamma'$, respectively.\end{proposition}

\begin{proof}[Sketch of proof]The first observation is that since $H'$, $G'$ and $\Gamma'$ are finite groups, then the covers $\psi_i$ actually descend to \textsc{Galois} covers of connected $\spec(\caa)$-schemes $\psi_{i,\caa}:X_{i,\caa}\to X_{i-1,\caa}$, where $\caa\subset k[[t]]$ is a regular finite type $k[t]$-algebra. Moreover, each $\psi_{i,\caa}$ still has the same group as $\psi_i$. Let $E:=\spec(\caa[t^{-1}])$. By taking the base changes $X_{i,E}:=X_{i,\caa}\times_{\spec(\caa)}\spec(E)$ and considering the induced morphisms $\psi_{i,E}$, they still have the same groups. Note that $X_{0,E}\cong Y_0\times_{\spec(k)}E$. Furthermore, for each $i=1,\cdots,r$, denote by $\zeta_{i,E}$ the point induced by $\zeta_i$ in $X_{0,E}$ for each $i=1,\cdots,r$ and observe that $\psi_{3,E}\circ\psi_{2,E}\circ\psi_{1,E}$ is only ramified over $\{\zeta_{1,E},\cdots,\zeta_{r,E}\}$. The proof is completed by showing that there exists an non-empty open subset $E'$ of $E$ so that the fiber of $\psi_{i,E}$ over each closed point of $E'$ is irreducible and non-empty. \end{proof}

It follows from Proposition \ref{ku6.6} that there exist smooth connected
projective curves $W^s$, $W^s_{XT}$ and $T^s$ defined over $k$
such that we have a sequence of coverings with the same group and
ramification behavior as (\ref{genpic})
\begin{equation}\label{spepic}
\xymatrix{
W^s\ar@/_/[rdd]_{\phi_{W^s},{\Gamma'}}\ar@{=}[r]&W^s\ar[d]^{H',\text{ \'etale}}\\
&W^s/H'\cong W_{XT}^s\ar[d]^{{G'},\text{ branched only at points lying over
  }\cau_{\Upsilon}}\\
T':=\phi_{T^s}^{-1}(U_{\Upsilon})\ \ar@{>->}[r]&T^s\ar[d]^{\phi_{T^s},\cag',\text{
    branched only at }
\cau_{\Upsilon}}\\
&\ov{U}_{\Upsilon}.}
\end{equation}
The ${\Gamma'}$-\textsc{Galois} cover
$\phi_{W^s}$ (which is \'etale over $T'$) gives a proper solution
$\lambda'_{T'}$ of
the embedding problem
\begin{equation}\label{cdt1a}
\xymatrix{
&\pi_1(T')\ar@{>>}[dl]_{\lambda'_{T'}}\ar@{>>}[d]^{\varphi'_{T'}}\\
{\Gamma'}\ar@{>>}[r]^{\alpha}&{G'}.
}
\end{equation}

\begin{remark}The diagram (\ref{spepic}) [in its lower part] shows that $T'$ (defined as the pre-image of $U_{\Upsilon}$ under the $\cag'$-\textsc{Galois} cover $\phi_{T^s}:T^s\to\ov{U}_{\Upsilon}$) is indeed an \'etale $\cag'$-\textsc{Galois} cover of $U_{\Upsilon}$. Moreover, Remark 3.3 and Lemma 2.8 assure us that $\cag'$ is a finite quotient of $M_p$. Hence, it belongs to the family of groups described in 3.2.
\end{remark}

\subsection{Solving (\ref{epn}) for $N$}\label{gpthar}.

As $T'$ is a curve, $\pi_1(T')$ is projective and so by Lemma \ref{split} we have a proper solution $\lambda_{T'}$ of the following embedding problem :
\begin{equation}\label{cdt1}
\xymatrix{
&\pi_1(T')\ar@{>>}[dl]_{\lambda_{T'}}\ar@{>>}[d]^{\varphi_{T'}}\\
{\Gamma}\ar@{>>}[r]^{\alpha}&{G}.
}
\end{equation}
Thus we are now in the
situation described by the following diagram
\begin{equation}\label{finaldiag}
\xymatrix{
&&N\
\ar@{>>}[d]_{\psi}\ar@{>->}[r]&\pi_1(T')\ar@{>>}[dl]_{\varphi_{T'}}\ \ar@{>->}[r]
&\pi_1(C)\ar@{>>}[dll]^{\varphi}\\
&\Gamma\ar@{>>}[r]^{\alpha}&G, }
\end{equation}
where $\Ker(\alpha)$ is minimal normal and is contained in
$M(\Gamma)$. Denote $\Psi:=(\lambda_{T'})_{|N}$. Observe that
$\alpha(\lambda_{T'}(N))=\varphi_{T'}(N)=\psi(N)=G$. Whence
$\lambda_{T'}(N)\cdot M(\Gamma)=\Gamma$, because $H\le M(\Gamma)$.
Since $\lambda_{T'}(N)$ is a normal subgroup of $\Gamma$, we
conclude from \cite[Proposition 8.3.6]{riza} that
$\Psi(N)=\Gamma$, \emph{i.e.}, $\Psi$ is a proper solution to the
embedding problem (\ref{epn}).

\subsection{The case where $p$ divides the order of $H$}\label{p1sec}

Since $H$ is minimal normal in $\Gamma$, then $H$ is generated by
its Sylow-$p$ subgroups, i.e.,  $H$ is a \emph{finite quasi-$p$ group}.
It follows from \cite[Theorem B]{po} or \cite[Corollary 4.6]{hacr}
that every finite quasi-$p$ embedding problem of smooth connected
affine curves defined over an algebraically closed field of
characteristic $p$ has a proper solution. Now, it suffices to
repeat the argument of Subsection \ref{gpthar}. \hfill $\square$

\section{Proof of Theorem \ref{mainthm} and group theoretic consequences}
In this section we obtain group theoretic consequences from
Theorems \ref{melnikov} and \ref{homogeneous}. In particular we
prove in Corollary \ref{propernormal} that given a normal subgroup
$N$ of $\pi_1(C)$, every proper open subgroup of $N$ is a free
profinite group.  As a consequence,  we obtain Kumar's result
\cite[Theorem 4.1]{ku} showing that the commutator subgroup of
$\pi_1(C)$ is a free profinite group of countable rank (see
Corollary \ref{kumar}).

\begin{definition}
A  profinite group is said to be  {\it virtually free}, if it has
an open free subgroup. \end{definition}

\begin{theorem}\label{characterization}
Let $C$, resp. $C_1$, resp. $C_2$, be smooth connected
  affine curves defined
  over an algebraically closed field of characteristic $p>0$ of
  countable cardinality. Let
  $\pi_1(C)$, resp. $\pi_1(C_1)$, resp. $\pi_1(C_2)$, be their
  respective algebraic fundamental groups. Let $N$ be a normal
   subgroup of
$\pi_1(C)$ such that a Sylow-$p$ subgroup $M_p$ of the quotient
group $M=\pi_1(C)/N$ is infinitely generated. Then
\begin{enumerate}
\item $N$  is a virtually free profinite group. \item Let
$G_1\triangleleft \pi_1(C_1)$ and $G_2\triangleleft\pi_1(C_2)$ be
two normal subgroups  such that the quotients $\pi_1(C_1)/G_1$ and
$\pi_1(C_2)/G_2$ have infinitely generated Sylow-$p$ subgroups.
Then $G_1$ and $G_2$
 are isomorphic if and only if $G_1/M_S(G_1)\cong G_2/M_S(G_2)$ for
  every finite simple group $S$.
\item  $N$ is a free profinite group of countable rank if and only
if $N/M_S(N)$ is infinite for every finite simple group $S$.
\end{enumerate}
\end{theorem}

\begin{proof}
Theorem \ref{homogeneous} together with  \cite[Corollary
  8.5.8]{riza} implies
(1) and together with \cite[Theorem 8.5.2]{riza} implies (2). Item
(3) is a consequence of (2) and Theorem \ref{melnikov} (item (4)).
\end{proof}

The proof of Theorem \ref{mainthm} follows from Theorem
\ref{homogeneous} and the next two theorems which are proved
similarly as in \cite[Theorem 2.4, Corollary 2.5 and Theorem
2.6]{za}, respectively. For the convenience of the reader we
include their proofs, taking into account that algebraic fundamental groups under consideration here (in contrast with the previous paper) are those of smooth connected affine curves over algebraically closed fields of positive characteristic and countable cardinality.

\begin{theorem}
\label{normal} Let $\pi_1(C)$ be the algebraic fundamental group
of a smooth connected affine curve $C$ defined over an
algebraically closed field $k$ of characteristic $p>0$ of
countable cardinality. Let $N$ be  a  normal  subgroup of
$\pi_1(C)$ such that a Sylow-$p$ subgroup $M_p$ of the quotient
group $M =\pi_1(C)/N$ is infinitely generated. Then $N$ is
isomorphic to a normal subgroup of a free profinite group of
countable rank.
\end{theorem}

\begin{proof}
By Theorems \ref{homogeneous} and \ref{melnikov} (item (1)), $N$ is isomorphic to an accessible
subgroup of infinite index of a free profinite group of countable rank. Therefore,
by Theorem \ref{melnikov} (item (2)), we just need to show that $N/M_l(N)$ is infinite
or trivial for any prime $l$. Note that by definition $N/M_l(N)$ is the maximal
elementary abelian pro-$l$ quotient  of $N$ 

Suppose $N/M_l(N)$ is finite. Then $M_l(N)$ is open in $N$ and
thus there exists an open subgroup $\cam$ of $\pi_1(C)$ such
that $M_l(\cam)\cap N=M_l(N)$. The group $\cam$ is
isomorphic to  the algebraic fundamental group $\pi_1(U)$ of a
finite \'etale cover $U$ of $C$. Moreover, $N$ is the intersection
of such $\cam$'s. Put
$\cam^{(l)}:=\cam/R_l(\cam)$,
$\ov{N}:=NR_l(\cam)/R_l(\cam)$.

If $l=p$, then  $\cam^{(p)}$ is a free pro-$p$ group of
countable rank (\emph{cf.} \cite[Corollary 3.7]{hais}).
Moreover, as
$\ov{N}$ is a normal subgroup of $\cam^{(p)}$, it is a free
pro-$p$ group of countable rank, unless it is trivial (\emph{cf.}
\cite[Proposition 8.6.3]{riza}). Hence, $N/M_p(N)$ is infinite or
trivial, as needed.

Suppose $l\neq p$. Observe now that we can choose $\cam$
satisfying the property $M_l(\cam)\cap N=M_l(N)$ with index in
$\pi_1(C)$ as large as we wish. In particular, the genus $g$ of
the smooth compactification $\ov{U}$ of $U$ and the finite number
$r$ of points of $\ov{U}\setminus U$, may be taken arbitrarily
large. Let $\pi_1(U)\twoheadrightarrow\pi_1(U)^{(l)}$ be the quotient homomorphism to the maximal pro-$l$ quotient of $\pi_1(U)$. It is known from the last statement of \cite[Exp. XIII, Corollaire 2.12]{sga1} that the maximal pro-prime-to-$p$ quotient $\pi_1(U)^{(p')}$ of $\pi_1(U)$ is a free profinite group on $t=2g+r-1$ generators. Taking its maximal pro-$l$-quotient it coincides with $\pi_1(U)^{(l)}$. Hence the latter group is a free pro-$l$ of
rank $t=2g+r-1$. We deduce from all this reasoning an epimorphism
$\Upsilon\twoheadrightarrow\hat{F}_t$, where the latter group is a
free pro-$l$ group of rank $t$. This epimorphism factors through
$\cam^{(l)}$ and since $\ov{N}$ is a normal subgroup of
$\cam^{(l)}$ it has rank at least $t$. But $t$ can be chosen
larger than the minimal number of generators of $N/M_l(N)$, this
yields a contradiction.
\end{proof}

\begin{remark}\label{remsha}
In fact \cite[Corollary 3.7]{hais} proves a bit more than what was
stated above. Although we do not require it for our purposes, it follows from
the above that for every smooth affine connected variety
$X$ defined over an algebraically closed field $k$ of
characteristic $p>0$, the maximal pro-$p$ quotient
$\pi_1(X)^{(p)}$ of the algebraic fundamental group $\pi_1(X)$ of
$X$ is free in as many generators as the cardinality of $k$.

In the case of smooth connected projective curves $X$, it was
already known to \textsc{Shafarevi\u{c}} \cite[Theorem 2]{sha}
that $\pi_1(X)$ is free in $\gamma_X$ generators, where $\gamma_X$
denotes the $p$-rank of $X$ (\emph{i.e.}, the $\bbf_p$-dimension
of the $p$-torsion subgroup $J_X[p]$ of the Jacobian variety $J_X$
of $X$). The number $\gamma_X$ is bounded above by the genus of
$X$. For a proof of the same result in the spirit of \'etale
cohomology \emph{cf.} \cite[Theorem 1.9]{cr}.
\end{remark}

\begin{corollary}\label{propernormal}
Every proper open subgroup of $N$ is a free profinite group.
\end{corollary}

\begin{proof}
The result follows from the preceding theorem combined with
\cite[Theorem 8.7.1]{riza}.
\end{proof}

\begin{theorem}
Let $\pi_1(C)$ be the algebraic fundamental group of a smooth
connected affine curve $C$ defined over an algebraically closed
field $k$ of characteristic $p>0$ of countable cardinality. Let
$N$ be a  characteristic subgroup of $\pi_1(C)$ such that
there exists a Sylow-$p$ subgroup $M_p$ of $M=\pi_1(C)/N$ which is
infinitely generated. Then $N$ is  isomorphic to a
characteristic subgroup of a free profinite group of countable
rank.
\end{theorem}

\begin{proof}
Let $S$ be a finite simple group.  By Theorem 2.9 and Theorem \ref{melnikov} (item
(3)), we just need to show that $N/M_S(N)$ is either infinite or
trivial. If $S$ has a prime order $l$, then this is already
contained in the proof of Theorem \ref{normal}.

Suppose now that $N/M_S(N)$ is finite non-trivial. Then there
exists a proper open normal subgroup $\cam$ of $\pi_1(C)$
containing $N$ such that $M_S(\cam)\cap N=M_S(N)$. Note that
$\cam$ is isomorphic to  $\pi_1(U)$ for some finite \'etale
\textsc{Galois} cover $U$ of $C$.

Assume $S$ is a non-abelian simple group of order divisible by $p$
and let $\Gamma$ be a finite direct product of $S$. By the
simplicity of $S$ and the divisibility by $p$ condition, $S$ is
quasi-$p$, i.e., is generated by  its Sylow $p$-subgroups. Next
consider the embedding problem
$(\pi_1(U)\twoheadrightarrow\{1\},\alpha:\Gamma\twoheadrightarrow\{1\})$.
Since $\Ker(\alpha)=\Gamma$ is a quasi-$p$ group, then by
\cite[Theorem B]{po} or \cite[Corollary 4.6]{hacr}, the latter
embedding problem admits a proper solution
$\pi_1(U)\twoheadrightarrow\Gamma$. Therefore we also have an
epimorphism $\cam/M_S(\cam)\twoheadrightarrow\Gamma$.
\emph{A fortiori}, $\cam/M_S(\cam)$ cannot be finite.
Hence, $\cam/M_S(\cam)$ equals an infinite direct product
of copies of $S$, whence it does not have finite characteristic
subgroups, which yields a contradiction. Consequently,
$\#(N/M_S(N))$ is prime to $p$.

Observe now
that we can choose $\cam$ satisfying the property $M_S(\cam)\cap
N=M_S(N)$ with index in $\pi_1(C)$ as large as we wish. In particular,
the genus $g$ of the smooth compactification $\ov{U}$ of $U$ and the
finite number $r$ of points of $\ov{U}\setminus U$, may be taken
arbitrarily large. Again we consider the quotient map $\pi_1(U)\twoheadrightarrow\pi_1(U)^{(p')}$ to the maximal pro-prime-to-$p$ quotient of $\pi_1(U)$. It follows from the last statement of \cite[exp. XIII, corollaire
2.12]{sga1} that 
$\pi_1(U)^{(p')}$ is isomorphic to the maximal pro-prime-to-$p$ quotient of a
free profinite group of rank $2g+r-1$, so it is free of rank
$t=2g+r-1$. We deduce from all this reasoning an epimorphism
$\cam\twoheadrightarrow\hat{F}_t$, where the latter group is a
free pro-prime-to-$p$ group of rank $t$. Thus, $\hat{F}_t/M_S(\hat{F}_t)$ is a
quotient of $\cam/M_S(\cam)$. Since the order of $\hat{F}_t/M_S(\hat{F}_t)$ tends to
infinity as $t$ does, we may assume that the order of
$\Upsilon/M_S(\cam)$ is larger than the order of $N/M_S(N)$. But $N/M_S(N)$
is characteristic in $\cam/M_S(\cam)\cong\prod S$, so $N/M_S(N)=\{1\}$,
because a finite product $\prod S$ does not have any proper characteristic
subgroup.
\end{proof}

As a corollary we obtain the main result of Kumar \cite[Theorem 4.1]{ku}.

\begin{corollary}\label{kumar}
The commutator subgroup of $\pi_1(C)$ is a free
profinite group of countable rank.
\end{corollary}

\begin{proof}
Let $K$ be the commutator subgroup of $\pi_1(C)$. Since this group is
characteristic and $\pi_1(C)/K$ has an infinitely generated
Sylow-$p$ subgroup (\emph{cf.} \cite[Corollary 3.7]{hais}), by the preceding theorem it is isomorphic to a
characteristic subgroup of a free profinite group of countable
rank. It then follows from Theorem \ref{melnikov} (item (3)) that $K/M_S(K)$ is
either trivial
or infinite.  Hence, by Theorem \ref{characterization} (item (3)), it suffices
to show that $K/M_S(K)$ is non-trivial for every finite simple
group $S$. 

Suppose that $S$ is non-abelian. If the order of $S$ is divisible by $p$, then $S$ is a quasi-$p$ group and by \textsc{Abhyankar}'s conjecture (completely proved in \cite{ha}), $S$ is necessarily a quotient of $\pi_1(C)$. Hence the quotient $\pi_1(C)/M_S(\pi_1(C))$ is non-trivial. If $p$ does not divide the order of $S$, then after replacing $C$ by an \'etale \textsc{Galois} cover $C'$ of $C$ of $p$-power order whose fundamental group still contains $K$, and using the fact that $C'$ can be chosen of arbitrarily high genus, we may assume that $2g+r-1\ge2$, where $g$ is the genus of $C$, $r=\#(\ov{C}\setminus C)$ and $\ov{C}$ is a smooth compactification of $C$. Since $S$ can be generated by at most 2 elements, it follows that $S$ is indeed a quotient of $\pi_1(C)$ and therefore $\pi_1(C)/M_S(\pi_1(C))$ is non-trivial. In particular, the group $\pi_1(C)/M_S(\pi_1(C))$ is
a nontrivial quotient of $K$, so $K/M_S(K))$ is non-trivial. 

As already observed we may and will assume that $2g-2+r\ge2$. For
$S$ of prime order $l\ne p$, we observe that the maximal pro-$l$
quotient $\pi_1(C)^{(l)}$ is isomorphic to the maximal
pro-$l$-quotient of a free profinite group of rank $2g+r-1$, so it
is free pro-$l$ of rank $t=2g+r-1$. So taking an open subgroup
$\cam$ of $\pi_1(C)$ containing $K$ of sufficiently large
index we see that its maximal pro-$l$ quotient $\cam^{(l)}$
is not abelian, and therefore the image of $K$ there is
non-trivial (since $K$ contains the commutator of $\cam$). Similarly, when $l=p$, it follows from Remark \ref{remsha} that $\pi_1(C)^{(p)}$ is a free pro-$p$ group with rank $\aleph_0$. In particular, by choosing an open subgroup $\cam$ of $\pi_1(C)$ we arrive at the same conclusion (\emph{i.e.} that the image of $K$ in $\cam^{(p)}$ cannot be trivial).
\end{proof}

\begin{remark}The techniques of the proof of the above result can be used to prove a more general result.  Let $N$ be the  kernel of an epimorphism of $\pi_1(C)$ to an abelian profinite group having infinitely generated  $p$-Sylow.  Then $N$ is free.   Indeed, in this case $N$ satisfies the hypothesis of Theorem 2.9, hence it is homogenous.  It follows from Theorem 2.7(4) that in order to prove that $N$ is free, we need to show that $N/M_S(N)$ is infinite for every finite simple group $S$.  In the case where $S$ is abelian, this follows exactly as in Corollary 5.7 using Theorem 5.3 instead of Theorem 5.6 and the item (2) of Theorem 2.7 instead of item (3).  In the case where $S$ is non-abelian, suppose that $N/M_S(N)=\prod^nS$ is finite and denote $R: =\prod^{n+1}S$. As was observed in the proof of Corollary 5.7, after having replaced $\pi_1(C)$ by an open normal subgroup $\cam$ of $\pi_1(C)$ containing $N$, we may assume that $C$ has genus sufficiently large to allow the existence of an epimorphism $f:\pi_1(C)\twoheadrightarrow R$. Since $R$ is perfect and the quotient modulo $\pi_1(C)/N$ is abelian,  then $f(N)=R$. This  contradicts  $N/M_S(N)=\prod^nS$. Thus  $N/M_S(N)$ is infinite.  \end{remark}

Next we give an example that is covered by Theorem 2.9, but not covered by the results of \textsc{Kumar}.

\begin{example}\label{examku}
Let $p$ be a prime number and let $S$ be a simple non-abelian
group containing a non-trivial \textsc{Sylow} $p$-subgroup. Let
$M=(\prod S)\times C_p$ be a direct product of a countable direct
power of $S$ by a group of order $p$. Then $M$ is quasi-$p$ group
(\emph{i.e.}, it is generated by its \textsc{Sylow}
$p$-subgroups). Moreover, $M=\varprojlim_JM_J$, where $J$ ranges
over finite sets and $M_J=\prod_{J} S$ a finite direct power of
$S$. Since for $|J_1|>|J_2|$ the embedding problem
\begin{equation}
\xymatrix { &\pi_1(C) \ar@{>>}[d]^{\psi}\\
M_{J_1} \ar@{>>} [r]^{\alpha}&M_{J_2}.}
\end{equation}
admits a proper solution (because the kernel of $\alpha$ is a direct product
of a finite number of copies of $S$, and so it is a finite quasi $p$-group, \emph{cf.} \cite[Theorem B]{po}), the universal property of
the inverse limit provides an epimorphism $\eta_1:\pi_1(C)\twoheadrightarrow
M$. Thus we can write $M$ as a quotient $M=\pi_1(C)/N$ of $\pi_1(C)$ by a normal subgroup $N$. The
quotient group $M/C_p$ has an infinitely generated Sylow $p$-subgroup. The composition of the epimorphism $\eta_1$ with the natural quotient epimorphism $\eta_2:M\twoheadrightarrow M/C_p$, induces an epimorphism $\eta_3:\pi_1(C)\twoheadrightarrow M/C_p$. Therefore, $M/C_p$ can also be written as a quotient $\pi_1(C)/N_p$ of $\pi_1(C)$ by a normal subgroup $N_p$. It now follows from
Theorem \ref{homogeneous} that $N_p$ is homogeneous. Since $N$
is open in $N_p$, we conclude that it is indeed free (as actually
any proper open subgroup of $N_p$) by Corollary
\ref{propernormal}. \hfill $\square$
\end{example}
%\newpage

\section{Appendix}

We rewrite here the proof of Theorem 2.2 in \cite{za}, where the
embedding problem (\ref{defhomo}) is properly solved for a
projective normal subgroup $N$  of the algebraic fundamental group
$\Gamma$ of a smooth connected projective curve of genus $g>1$
defined over an algebraically closed field of characteristic zero.
Recall that this group is just the profinite completion of the
topological fundamental group $\Pi^{\text{top}}$ of a compact
\textsc{Riemann} surface of genus $g>1$ (\emph{cf.} \cite[Expos\'e
X, end of p. 208 and beginning of p. 209]{sga1}). The group
$\Pi^{\text{top}}$ admits a presentation of the form
$\Pi^{\text{top}}=\langle x_i,y_i\mid \prod_{i=1}^g [x_i,y_i]=1,
i=1,\ldots, g\rangle$. Here we exploit this presentation; in
particularly we rely heavily on the relation $\prod_{i=1}^g
[x_i,y_i]=1$. In the case of positive characteristic an explicit
presentation for $\pi_1(C)$ is not available. Therefore, in order
to get (as we did) a similar statement in positive characteristic
(for smooth connected affine curves), it was necessary to use
methods from algebraic geometry  which allow us to construct
solutions for embedding problems.

The present proof simplifies
the proof of \cite{za} and eliminates small but embarrassing
errors. Note also that we do not need here splitting of the
embedding problem.

\begin{lemma}\label{appropriate}
Let $\Gamma=\langle x_i,y_i\mid \prod_{i=1}^g [x_i,y_i]\rangle$ be
the the algebraic fundamental group of a smooth connected
projective curve of genus $g>1$ defined over an algebraically
closed field of characteristic zero. Consider the following finite
embedding problem:
\begin{equation}\label{apeq1}
\xymatrix { &\Gamma \ar@{>>}[d]^f\\
              A \ar@{>>} [r]^{\alpha}&B.  }
\end{equation}
Suppose it admits a weak solution
$\varphi:\Gamma\longrightarrow A$ such that
$\varphi(x_1)=\varphi(x_{2})=\cdots =\varphi(x_{sn+s})$ and
$\varphi(y_1)=\varphi(y_{2})=\cdots =\varphi(y_{sn+s})$, where
$n=|K \varphi(\Gamma)|$  and $s$ is the minimal number of
generators of $K=ker(\alpha)$. Then (\ref{apeq1}) admits a proper
solution.
\end{lemma}

\begin{proof}
We use the notation  $[x,y]$ for $x^{-1}y^{-1}xy$ and $\bar g$ for
$\varphi(g)$ in the argument that follows. Choose a minimal set of
generators $k_1,\ldots, k_s$ of $K$. Let $\eta$ be a map that
sends $x_{in+1},x_{in+2},\ldots x_{in+n}$ to $\bar
x_{sn+i+1}k_{i+1}$ for all $i=0,\ldots s-1$ and coincides with
$\varphi$ on the other generators. Then $\eta$ extends to a
homomorphism if $$ [\eta(x_1),\eta(y_1)]\ldots
[\eta(x_{g}),\eta(y_{g})]=1
$$
(since this would mean that the homomorphism from a free profinite
group\break $F(x_1,y_1,\ldots,x_{g}, y_{g})\longrightarrow A$
extending $\eta$ factors through $\Gamma$).

\bigskip \noindent Since $[\bar x_{in+1}k_i,\bar y_{in+1}]\in
K\varphi(\Gamma)$ one has $([\bar x_{in+1}k_i,\bar
y_{in+1}])^n=1=([\bar x_{in+1},\bar y_{in+1}])^n.$ Therefore
$[\eta(x_1),\eta(y_1)]\cdots
[\eta(x_{sn}),\eta(y_{sn})]=\prod_{i=0}^{s-1}([\bar
x_{in+1}k_i,\bar y_{in+1}])^n =1$ as well as $[\bar x_1,\bar
y_1]\cdots [\bar x_{sn},\bar y_{sn}]=\prod_{i=0}^{s-1}([\bar
x_{in+1},\bar y_{in+1}])^n=1$. But
$$[\eta(x_{sn+1}),\eta(y_{sn+1})]\cdots
[\eta(x_{g}),\eta(y_{g})]=[\bar x_{sn+1},\bar y_{sn+1}]\cdots
[\bar x_{g},\bar y_{g}].$$  It follows that
$[\eta(x_1),\eta(y_1)]\cdots [\eta(x_{g}),\eta(y_{g})]=[\bar
x_1,\bar y_1]\cdots [\bar x_{g},\bar y_{g}]=1$ as needed.

\bigskip\noindent Thus $\eta$ extends to a homomorphism $\psi\colon
\Gamma\longrightarrow A$ such that  $f=\alpha\psi$. But
$$\psi(x_1^{-1}x_{sn+1})=k_1, \ldots,
\psi(x_{n(s-1)+1}^{-1}x_{ns+s})=k_s$$ so $\psi$ is an epimorphism
and the lemma is proved.\end{proof}

\noindent The next lemma shows that by replacing $\Gamma$ with an
open subgroup of sufficiently large index, we can reduce ourselves
to the situation of Lemma \ref{appropriate}.

\begin{lemma}\label{reduction} Let $N$ be
a  projective subgroup of $\Gamma$ and consider the embedding problem
\begin{equation}
\xymatrix { &N \ar@{>>}[d]^f\\
              A \ar@{>>} [r]^{\alpha}&B,  }
\end{equation}
where $A$, $B$ are finite. Then there exists an open  subgroup $U$
of $\Gamma$ containing $N$ and an embedding problem
\begin{equation}
\xymatrix { &U \ar@{>>}[d]^{\eta}\\
              A \ar@{>>} [r]^{\alpha}&B,  }
\end{equation}
satisfying hypothesis of Lemma \ref{appropriate} such that the
restriction $\eta_{|N}=f$.\end{lemma}

\begin{proof} Since $N$ is projective there exists a
homomorphism $f':N\longrightarrow A$ such that
$\alpha\varphi(N)=B$. Put $B'=f'(N)$.

By Lemma {8.3.8}  in [RZ-2000] there exists  an open  subgroup $U$
of $\Gamma_g$ containing $N$ and an epimorphism
$\varphi:U\longrightarrow B'$ such that  $\varphi_{|N}=f'$. Since
an open subgroup of $\Gamma_g$ is again a profinite surface group,
replacing $\Gamma_g$ by $U$ we may assume the existence of the
following commutative diagram:
\begin{equation}
\xymatrix {& N \ar@{>>}[dd]_(0.6)f\ar[r]\ar[dl]_{f'}&\Gamma_g\ar@{>>}[ddl]^{\varphi_0}\ar[dll]^(0.7){\varphi}\\
           B'\ar[d]\ar[dr]&&\\
              A \ar@{>>} [r]^{\alpha}&B,&  }
\end{equation}
where the top horizontal map is the natural inclusion.
 Let $U_i$ be the family of all open  subgroups of $\Gamma_g$ containing $N$.
Then $\varphi_i:=\varphi_{|U_i}$ is an epimorphism for every $i$.
Note that every $U_i$ is again a profinite surface group and so
has a presentation $U_i=\langle x_1,y_1,\ldots x_{g_i},
y_{g_i}\mid \prod_{j=1}^{g_i}[x_i,y_i]\rangle$, where the genus
$g_i$ of $U_i$ can be computed by the formula
$g_i-1=[\Gamma_g:U_i](g-1)$ (we abusing notation of course, for
different $i$, the elements $x_1,y_1\ldots$ are not the same and
not the same as those in $\Gamma$). This means that we can choose
$i$ with the number of generators of $U_i$ sufficiently large, so
that there exists $i$ with $\varphi(x_{j_1})=\cdots
=\varphi(x_{j_{sn+s}})$ and $\varphi(y_{j_1})=\cdots
=\varphi(y_{j_{sn+s}})$, where $n=|K||B'|$, $s$ is the minimal
number of generators of $K$ and $j_l>j_k$ whenever $l>k$. We shall
use the notation $x^y$ for $y^{-1}xy$ in the argument to follow.
Suppose $j_1\neq 1$. Then $\prod_{j=1}^{g_i}
[x_j,y_j]=[x_{j_1},y_{j_1}]([x_1,y_1]\cdots
[x_{j_1-1},y_{j_1-1}])^{[x_{j_1},y_{j_1}]}[x_{j_1+1},y_{j_1+1}]\cdots
[x_{g_i},y_{g_i}]\,$ so replacing the generators $x_1, y_1,\ldots,
x_{j_1-1},y_{j_1-1}$ by
$x_1^{[x_{j_1},y_{j_1}]},y_1^{[x_{j_1},y_{j_1}]},\ldots
x_{j-1}^{[x_{j_1},y_{j_1}]},y_{j-1}^{[x_{j_1},y_{j_1}]}$ we may
assume that $j_1=1$. Continuing similarly, we in fact may assume
that $\varphi(x_1)=\varphi(x_{2})=\cdots =\varphi(x_{sn+s})$ and
$\varphi(y_1)=\varphi(y_{2})=\cdots
=\varphi(y_{sn+s})$.\end{proof}

 We now solve the embedding problem for  an accessible projective subgroup $N$, having $N$
normal as a particular case. Note that only in this proof we assume the
minimality of the kernel $K$.

\begin{theorem}\label{char0}  A projective accessible subgroup  $N$ of
$\Gamma$ is  isomorphic to an accessible subgroup of a free
profinite group.\end{theorem}

\bigskip

\begin{proof}
  By Theorem  \ref{melnikov} we just need to solve the following embedding problem for $N$:
\begin{equation}
\xymatrix { &N \ar@{>>}[d]^f\\
              A \ar@{>>} [r]^{\alpha}&B,  }
\end{equation}
where $A$, $B$ are finite, $K:={\rm Ker}(\alpha)$ is minimal
normal and $K\leq M(A)$.

By two preceding lemmas we have an open subgroup $N\leq U\leq
\Gamma$ and an epimorphism  $\psi:U\longrightarrow A$ such that
$\alpha(\psi(N)=B$ and so $\psi(N)M(A)=A$. Since $\psi(N)$ is a
subnormal subgroup of $A$ by Proposition 8.3.6 in [RZ] $\psi(N)=A$
as needed.
\end{proof}

\end{document}